\documentclass[11pt]{article}
\usepackage{amssymb,amsmath}
\usepackage{mathrsfs}
\usepackage{latexsym}
\usepackage{color}
\usepackage{amssymb}
\usepackage{amsmath}
\usepackage{amsthm}

\allowdisplaybreaks

\catcode`!=11
\let\!int\int \def\int{\displaystyle\!int}
\let\!lim\lim \def\lim{\displaystyle\!lim}
\let\!sum\sum \def\sum{\displaystyle\!sum}
\let\!sup\sup \def\sup{\displaystyle\!sup}
\let\!inf\inf \def\inf{\displaystyle\!inf}
\let\!cap\cap \def\cap{\displaystyle\!cap}
\let\!max\max \def\max{\displaystyle\!max}
\let\!min\min \def\min{\displaystyle\!min}
\catcode`!=12

\newtheorem{theorem}{\bf Theorem}[section]
\newtheorem{lemma}{\bf Lemma}[section]
\newtheorem{definition}{\bf Definition}[section]
\newtheorem{proposition}{\bf Proposition}[section]
\newtheorem{corollary}{\bf Corollary}[section]
\newtheorem{remark}{\bf Remark}[section]

\catcode`@=11
\@addtoreset{equation}{section}
\catcode`@=12

\begin{document}
\title{Strong averaging principle for stochastic Klein-Gordon equation
with a fast oscillation
\footnote{This work is supported by NSFC Grant (11601073) and the Fundamental Research Funds for the Central Universities}}
\author{Peng Gao
\\[2mm]
\small School of Mathematics and Statistics, and Center for Mathematics
\\
\small and Interdisciplinary Sciences, Northeast Normal University,
\\
\small Changchun 130024,  P. R. China
\\[2mm]
\small Email: gaopengjilindaxue@126.com }
\date{}
\maketitle

\vbox to -13truemm{}

\begin{abstract}
This paper investigates an averaging principle for stochastic Klein-Gordon equation
with a fast oscillation arising as the solution of a stochastic reaction-diffusion equation evolving with
respect to the fast time. Stochastic averaging principle is a powerful tool for studying qualitative analysis of stochastic dynamical systems
with different time-scales. To be more precise, the well-posedness of mild solutions of the stochastic hyperbolic-parabolic equations is firstly established
by applying the fixed point theorem and the cut-off technique.  Then, under suitable conditions, we prove that there is a limit process in which the fast varying process
is averaged out and the limit process which takes the form of the stochastic Klein-Gordon equation is an average with respect to the stationary measure of the fast varying process.
Finally, by using the Khasminskii technique we can obtain the rate of strong convergence for the slow component towards the solution of the averaged equation.
\\[6pt]
{\sl Keywords: Stochastic averaging principle; Stochastic Klein-Gordon equation; Effective dynamics; slow-fast SPDEs; Strong convergence}
\\
\\
{\sl 2010 Mathematics Subject Classification: 60H15, 70K65, 70K70}
\end{abstract}
\section{Introduction}
\par
The nonlinear Klein-Gordon equation
\begin{equation*}
\begin{array}{l}
\begin{array}{llll}
u_{tt}-u_{xx}+m^{2}u+\mu|u|^{2}u+\nu|u|^{4}u=0,
\end{array}
\end{array}
\begin{array}{lll}
\end{array}
\end{equation*}
appears in the study of several problems of mathematical physics.
For example, this equation arises in general relativity, nonlinear optics (e.g., the instability phenomena such as self-focusing), plasma physics, fluid mechanics, radiation theory or spin waves \cite{D3,F8,L4}.
\par
Stochastic Klein-Gordon equation is a stochastic wave equation, a large amount of work has been devoted to the study of the nonlinear stochastic wave equation:
\par
$\bullet$ Existence and uniqueness of solution: \cite{G5} establishs the existence and uniqueness
of solution for stochastic viscoelastic wave equations.

\par
$\bullet$ Explosive solution:
\cite{C6}, \cite{T1} and \cite{B7} invtisvities the explosive solution of stochastic
wave equation.
\par
$\bullet$ Large-time asymptotic properties of solutions:
Large-time asymptotic properties of solutions to a class of semilinear stochastic
wave equations with damping in a bounded domain are considered in \cite{C7}.
In \cite{L5}, relations between the asymptotic behavior for a stochastic wave equation and a heat equation
are considered.
\par
$\bullet$ Absolute continuity of the law of the solution: In \cite{Q1}, the authors prove some results concerning the existence of the density of the real valued solution of a 3D-stochastic wave equation.

\par
$\bullet$ Invariant measure: The existence and uniqueness of an invariant measure for the transition semigroup associated with a nonlinear stochastic Klein-Gordon type are studied in \cite{B1} and \cite{B5}, in \cite{B5}, the authors consider the stochastic wave equations with nonlinear dissipative damping. In \cite{B6}, the authors show the existence of a unique invariant measure associated with the transition
semigroup under mild conditions.
\par
$\bullet$ The corresponding Kolmogorov operator: In \cite{B1}, the structure of
the corresponding Kolmogorov operator associated with a stochastic Klein-Gordon equation is studied.
\par
$\bullet$ Attractor: In \cite{C5}, the existence of an attractor is proved,
which implies the existence of an invariant measure. However, there is no a large overlap
with the results obtained here and the methods are quite different. \cite{W2} deals with
a class of non-autonomous stochastic linearly damped wave equations on Rd perturbed by multiplicative Stratonovich white noise of the form.

\par $\bullet$ Smoluchowski-Kramers approximation problem: The Smoluchowski-Kramers approximation problem for the nonlinear stochastic wave equation has been consider in \cite{C8,C9,C10,C11,C12,C13}.

\par $\bullet$ Large deviation principle. In \cite{L6}, by using a weak convergence method, a large deviation principle is built for the singularly perturbed stochastic nonlinear damped wave equations on bounded regular domains.
\par
$\cdots\cdots$
\par
In this paper, we will be concerned with the averaging principle for stochastic Klein-Gordon
equation with a fast oscillating perturbation
\begin{equation*}
\begin{array}{l}
\left\{
\begin{array}{llll}
dA^{\varepsilon}_{t}+[-A^{\varepsilon}_{xx}+\mu|A^{\varepsilon}|^{2}A^{\varepsilon}+\nu|A^{\varepsilon}|^{4}A^{\varepsilon}+f(A^{\varepsilon}(t), B(\frac{t}{\varepsilon}))]dt=\sigma_{1}dW_{1}

\\A^{\varepsilon}(0,t)=0=A^{\varepsilon}(1,t)

\\A^{\varepsilon}(x,0)=A_{0}(x)
\\A^{\varepsilon}_{t}(x,0)=A_{1}(x)

\end{array}
\right.
\end{array}
\begin{array}{lll}
{\rm{in}}~Q\\

{\rm{in}}~(0,T)\\

{\rm{in}}~I\\

{\rm{in}}~I,
\end{array}
\end{equation*}
where $B(t)$ is governed by the stochastic reaction-diffusion equation
\begin{equation*}
\begin{array}{l}
\left\{
\begin{array}{llll}
dB+[-B_{xx}+|B|^{2}B+g(A, B)]dt=\sigma_{2}dW_{2}
\\B(0,t)=0=B(1,t)
\\B(x,0)=B_{0}(x)
\end{array}
\right.
\end{array}
\begin{array}{lll}
{\rm{in}}~Q\\
{\rm{in}}~(0,T)\\
{\rm{in}}~I,
\end{array}
\end{equation*}
where $T>0, I=(0,1), Q=I\times (0,T),$ the stochastic perturbations are of additive type, $W_{1}$ and $W_{2}$ are
mutually independent Wiener processes on a complete stochastic basis $(\Omega,\mathcal{F},\mathcal{F}_{t},\mathbb{P})$, which
will be specified later, denote by $\mathbb{E}$ the expectation with respect to $\mathbb{P}$. The coefficients $\mu$ and $\nu$ are positive constants, the noise coefficients $\sigma_{1}$ and $\sigma_{2}$ are positive constants.
\par
Thus, we will be concerned with the averaging principle for multiscale stochastic Klein-Gordon
equation with slow and fast time-scales
\begin{equation*}
\begin{array}{l}
(\ast)\left\{
\begin{array}{llll}
dA^{\varepsilon}_{t}+[-A^{\varepsilon}_{xx}+\mu|A^{\varepsilon}|^{2}A^{\varepsilon}+\nu|A^{\varepsilon}|^{4}A^{\varepsilon}+f(A^{\varepsilon}, B^{\varepsilon})]dt=\sigma_{1}dW_{1}
\\dB^{\varepsilon}+\frac{1}{\varepsilon}[-B^{\varepsilon}_{xx}+|B^{\varepsilon}|^{2}B^{\varepsilon}+g(A^{\varepsilon}, B^{\varepsilon})]dt=\frac{1}{\sqrt{\varepsilon}}\sigma_{2}dW_{2}
\\A^{\varepsilon}(0,t)=0=A^{\varepsilon}(1,t)
\\B^{\varepsilon}(0,t)=0=B^{\varepsilon}(1,t)
\\A^{\varepsilon}(x,0)=A_{0}(x)
\\A^{\varepsilon}_{t}(x,0)=A_{1}(x)
\\B^{\varepsilon}(x,0)=B_{0}(x)
\end{array}
\right.
\end{array}
\begin{array}{lll}
{\rm{in}}~Q\\
{\rm{in}}~Q\\
{\rm{in}}~(0,T)\\
{\rm{in}}~(0,T)\\
{\rm{in}}~I\\
{\rm{in}}~I\\
{\rm{in}}~I,
\end{array}
\end{equation*}
where the parameter
$\varepsilon$ is small and positive, which describes the ratio of time scale between the process $A^{\varepsilon}$ and
$B^{\varepsilon}$. With this time scale the variable $A^{\varepsilon}$ is referred as slow component and $B^{\varepsilon}$ as the
fast component.

\par
The theory of stochastic averaging principle provides an effective approach for the
qualitative analysis of stochastic systems with different time-scales and is relatively
mature for stochastic dynamical systems.
The theory of averaging principle serves as a tool in study of the qualitative behaviors for complex systems with multiscales, it is essential for describing and understanding the asymptotic behavior of dynamical systems with fast and slow variables. Its basic idea is to approximate the original
system by a reduced system.
The averaging principle is an important method to extract effective macroscopic dynamic from complex systems with slow component and fast component.The theory of averaging for deterministic dynamical systems, which
was first studied by Bogoliubov \cite{B0}, has a long and rich history.

\par
The averaging principle in the stochastic ordinary
differential equations setup was first considered by Khasminskii \cite{K1} which proved that
an averaging principle holds in weak sense, and has been an active research field on which there is a great deal of literature.
Recently, the averaging principle for stochastic differential equations has been paid
much attention \cite{F5,F6,G2,G3,L2}.
\par
However, there are few results on the averaging principle for stochastic systems in infinite dimensional space. To this
purpose we recall the recent results:
\par~~
\par
$\bullet$ parabolic-parabolic system: Cerrai and Freidlin \cite{C1}, Cerrai \cite{C2,C3}, Br\'{e}hier \cite{B2}, Wang and Roberts \cite{W1}, Fu and co-workers \cite{F1,F2,F4}, Xu and co-workers \cite{X1,X2}, Bao and co-workers \cite{B4};
\par~~
\par
$\bullet$ hyperbolic-parabolic system: Fu and co-workers \cite{F1,F7}, Pei and co-workers \cite{P3};
\par~~
\par
$\bullet$ Burgers-parabolic system: Dong and co-workers \cite{D2};
\par~~
\par
$\bullet$ FitzHugh-Nagumo system: Fu and co-workers \cite{F3}, Xu and co-workers \cite{X1}.
\par~~

\par
However, as far as we know there are
no results on the averaging principle for the stochastic Klein-Gordon equations with a fast oscillation $(\ast)$, a natural question is as follows:
\par
~~
\par \textit{Can we establish the averaging principle for the stochastic Klein-Gordon equations with a fast oscillation $(\ast)$ ? To be more precise, can the slow component $A^{\varepsilon}$  be approximated by the solution $\bar{A}$
which governed by a stochastic Klein-Gordon equation?
}
\par
~~
\par
These mathematical questions arise naturally which are important
from the point of view of dynamical systems from both physical and mathematical standpoints. In this paper, the main object is to establish an
effective approximation for slow process $A^{\varepsilon}$ with respect to the limit $\varepsilon\rightarrow0$.
\par
In this paper, we will take
\begin{equation*}
\begin{array}{l}
\begin{array}{llll}
\mu=\nu=1
\end{array}
\end{array}
\end{equation*}
for the sake of simplicity. All the results
can be extended without difficulty to the general case.
\par
We define
\begin{equation*}
\begin{array}{l}
\begin{array}{llll}
\mathcal{L}(u)=u_{xx},
\\
\mathcal{F}(u)=-|u|^{2}u-|u|^{4}u,
\\
\mathcal{G}(u)=-|u|^{2}u,
\end{array}
\end{array}
\end{equation*}
then the stochastic Klein-Gordon
equation $(\ast)$ becomes
\begin{equation}\label{1}
\begin{array}{l}
\left\{
\begin{array}{llll}
dA^{\varepsilon}_{t}=[\mathcal{L}(A^{\varepsilon})+\mathcal{F}(A^{\varepsilon})+f(A^{\varepsilon}, B^{\varepsilon})]dt+\sigma_{1}dW_{1}
\\dB^{\varepsilon}=\frac{1}{\varepsilon}[\mathcal{L}(B^{\varepsilon})+\mathcal{G}(B^{\varepsilon})+g(A^{\varepsilon}, B^{\varepsilon})]dt+\frac{1}{\sqrt{\varepsilon}}\sigma_{2}dW_{2}
\\A^{\varepsilon}(0,t)=0=A^{\varepsilon}(1,t)
\\B^{\varepsilon}(0,t)=0=B^{\varepsilon}(1,t)
\\A^{\varepsilon}(x,0)=A_{0}(x)
\\A^{\varepsilon}_{t}(x,0)=A_{1}(x)
\\B^{\varepsilon}(x,0)=B_{0}(x)
\end{array}
\right.
\end{array}
\begin{array}{lll}
{\rm{in}}~Q,\\
{\rm{in}}~Q,\\
{\rm{in}}~(0,T),\\
{\rm{in}}~(0,T),\\
{\rm{in}}~I,\\
{\rm{in}}~I,\\
{\rm{in}}~I.
\end{array}
\end{equation}

\par
Multiscale stochastic partial differential equations arise as models for various complex
systems, such model arises from describing multiscale phenomena in, for example, nonlinear oscillations,
material sciences, automatic control, fluids dynamics, chemical kinetics and in other areas leading to mathematical description
involving ``slow'' and ``fast'' phase variables. The study of the asymptotic
behavior of such systems is of great interest. In this respect, the question of how the physical effects at large time scales
influence the dynamics of the system is arisen. We focus on this question and show
that, under some dissipative conditions on fast variable equation, the complexities effects at
large time scales to the asymptotic behavior of the slow component can be omitted or neglected
in some sense.
\subsection{Mathematical setting}
\par We introduce the following mathematical setting:

\par
$\diamond$ We denote by $L^{2}(I)$ the space of all Lebesgue square integrable functions on $I$. The inner product on $L^{2}(I)$ is
\begin{eqnarray*}
( u,v)=\int_{I}uvdx,
\end{eqnarray*}
for any $u,v\in L^{2}(I).$ The norm on $L^{2}(I)$ is
\begin{eqnarray*}
\|u\|=( u,u )^{\frac{1}{2}},
\end{eqnarray*}
for any $u\in L^{2}(I).$
\par
$H^{s}(I)(s\geq 0)$ are the classical Sobolev
spaces of functions on $I$. The definition of $H^{s}(I)$ can be found in \cite{L1}, the norm on $H^{s}(I)$ is
$\|\cdot\|_{H^{s}}.$
\par
We set
\begin{eqnarray*}
\begin{array}{l}
\begin{array}{llll}
X_{p,\tau}=L^{p}(\Omega;C([0,\tau];H^{1}(I))\times L^{p}(\Omega;C([0,\tau];L^{2}(I))\times L^{p}(\Omega;C([0,\tau];H^{1}(I)),\\
Y_{\tau}=C([0,\tau];H^{1}(I)\times C([0,\tau];L^{2}(I)\times C([0,\tau];H^{1}(I),
\end{array}
\end{array}
\end{eqnarray*}
where $p\geq1,\tau\geq0.$

\par
$\diamond$ For $i=1,2,$ let $\{e_{i,k}\}_{k\in \mathbb{N}}$ be eigenvectors of a nonnegative, symmetric operator $Q_{i}$ with corresponding eigenvalues $\{\lambda_{i,k}\}_{k\in \mathbb{N}}$, such that
\begin{eqnarray*}
\begin{array}{l}
\begin{array}{llll}
Q_{i}e_{i,k}=\alpha_{i,k}e_{i,k},~~\lambda_{i,k}>0,~k\in \mathbb{N}.
\end{array}
\end{array}
\end{eqnarray*}
Let $W_{i}$ be an $L^{2}(I)-$valued $Q_{i}$-Wiener process with operator $Q_{i}$ satisfying
\begin{eqnarray*}
\begin{array}{l}
\begin{array}{llll}
TrQ_{i}=\sum\limits_{k=1}^{+\infty}\alpha_{i,k}<+\infty,~~k\in \mathbb{N}
\end{array}
\end{array}
\end{eqnarray*}
and
\begin{eqnarray*}
\begin{array}{l}
\begin{array}{llll}
W_{i}=\sum\limits_{k=1}^{+\infty}\alpha_{i,k}^{\frac{1}{2}}\beta_{i,k}(t)e_{i,k}<+\infty,~~k\in \mathbb{N}~~t\geq0,
\end{array}
\end{array}
\end{eqnarray*}
where $\{\beta_{i,k}\}_{k\in \mathbb{N}}(i=1,2)$ are independent real-valued Brownian motions on the probability base $(\Omega,\mathcal{F},\mathcal{F}_{t},\mathbb{P})$.
\par
We denote $\|\sigma_{i}\|_{Q_{i}}^{2}\triangleq\sigma_{i}^{2}TrQ_{i}.$
\par
$\diamond$ The functions $f$ and $g$ satisfy the global Lipschitz condition and the sublinear growth condition, specifically, there exist positive
constants $L_{f}$ and $L_{g}$ such that
\begin{eqnarray*}
\begin{array}{l}
\begin{array}{llll}
\|f(u_{1},v_{1})-f(u_{2},v_{2})\|\leq L_{f}(\|u_{1}-u_{2}\|+\|v_{1}-v_{2}\|),
\\
\|g(u_{1},v_{1})-g(u_{2},v_{2})\|\leq L_{g}(\|u_{1}-u_{2}\|+\|v_{1}-v_{2}\|)
\end{array}
\end{array}
\end{eqnarray*}
for all $u_{1},u_{2},v_{1},v_{2}\in L^{2}(I).$
\par
$\diamond$ Throughout the paper, the letter $C$ denotes positive constants
whose value may change in different occasions. We will write the dependence
of constant on parameters explicitly if it is essential.

\par We adopt the following hypothesis (H) throughout this paper:
\par
(H) $\alpha\triangleq\lambda-L_{g}>0,$ where $\lambda>0$ is the smallest constant such that the following inequality holds
\begin{eqnarray*}
\begin{array}{l}
\begin{array}{llll}
\|u_{x}\|^{2}\geq\lambda\|u\|^{2},
\end{array}
\end{array}
\end{eqnarray*}
where $u\in H_{0}^{1}(I)$ or $\int_{I}udx=0.$

\subsection{Main results}

\par
Asymptotical methods play an important role in investigating
nonlinear dynamical systems. In particular, the averaging
methods provide a powerful tool for simplifying dynamical systems,
and obtain approximate solutions to differential equations
arising from mechanics, mathematics, physics, control and
other areas. In this paper, we use stochastic averaging principle to investigate
stochastic Klein-Gordon equation (\ref{1}).
\par
Now, we are in a position to present the main result in this paper.
\begin{theorem}\label{Th1}
Suppose that the hypothesis (H) holds and $A_{0},B_{0}\in H^{1}_{0}(I),A_{1}\in L^{2}(I),$ $(A^{\varepsilon},B^{\varepsilon})$ is the solution of (\ref{1}) and $\bar{A}$ is the solution of the effective dynamics equation
\begin{equation}\label{8}
\begin{array}{l}
\left\{
\begin{array}{llll}
d\bar{A}_{t}=[\mathcal{L}(\bar{A})+\mathcal{F}(\bar{A})+\bar{f}(\bar{A})]dt+\sigma_{1}dW_{1}
\\\bar{A}(0,t)=0=\bar{A}(1,t)
\\\bar{A}(x,0)=A_{0}(x)
\\\bar{A}_{t}(x,0)=A_{1}(x)
\end{array}
\right.
\end{array}
\begin{array}{lll}
{\rm{in}}~Q\\
{\rm{in}}~(0,T)\\
{\rm{in}}~I,\\
{\rm{in}}~I,
\end{array}
\end{equation}
then we have for any $T>0,$ any $p>0,$
\begin{eqnarray*}
\begin{array}{l}
\begin{array}{llll}
\lim\limits_{\varepsilon\rightarrow 0}(\mathbb{E}\sup\limits_{0\leq t\leq T}\|A^{\varepsilon}(t)-\bar{A}(t)\| ^{2p}+\mathbb{E}\sup\limits_{0\leq t\leq T}\|A^{\varepsilon}_{t}(t)-\bar{A}_{t}(t)\| ^{2p})=0,
\end{array}
\end{array}
\end{eqnarray*}
where
\begin{eqnarray*}
\begin{array}{l}
\begin{array}{llll}
\bar{f}(A)=\int_{L^{2}(I)}f(A,B)\mu^{A}(dB)
\end{array}
\end{array}
\end{eqnarray*}
and $\mu^{A}$ is an invariant measure for the fast motion with frozen slow component
\begin{equation}\label{3}
\begin{array}{l}
\left\{
\begin{array}{llll}
dB=[\mathcal{L}(B)+\mathcal{G}(B)+g(A,B)]dt+\sigma_{2}dW_{2}
\\B(0,t)=0=B(1,t)
\\B(x,0)=B_{0}(x)
\end{array}
\right.
\end{array}
\begin{array}{lll}
{\rm{in}}~Q\\
{\rm{in}}~(0,T)\\
{\rm{in}}~I,
\end{array}
\end{equation}
where $A\in L^{2}(I).$
\par
\par
Moreover, if $p>\frac{5}{8},$ there exists a positive constant $C(p)$ such that
\begin{equation*}
\begin{array}{l}
\begin{array}{llll}
\mathbb{E}\sup\limits_{0\leq t\leq T}\|A^{\varepsilon}(t)-\bar{A}(t)\| ^{2p}+\mathbb{E}\sup\limits_{0\leq t\leq T}\|A^{\varepsilon}_{t}(t)-\bar{A}_{t}(t)\| ^{2p}
\leq C(p)(\frac{1}{-\ln \varepsilon})^{\frac{1}{8p}};
\end{array}
\end{array}
\end{equation*}
if $0<p\leq\frac{5}{8},$ for any $\kappa>0,$ there exists a positive constant $C(p,\kappa)$ such that
\begin{eqnarray*}
\begin{array}{l}
\begin{array}{llll}
\mathbb{E}\sup\limits_{0\leq t\leq T}\|A^{\varepsilon}(t)-\bar{A}(t)\| ^{2p}+\mathbb{E}\sup\limits_{0\leq t\leq T}\|A^{\varepsilon}_{t}(t)-\bar{A}_{t}(t)\| ^{2p}
\leq C(p,\kappa)(\frac{1}{-\ln \varepsilon})^{\frac{8p}{(5+4\kappa)^{2}}}

.
\end{array}
\end{array}
\end{eqnarray*}
\end{theorem}

\par
This paper is organized as follows.
In Sec. 2, we present
some preliminary results and an exponential ergodicity of a fast motion equation (\ref{3}) with the frozen slow
component. In Sec. 3, we establish the well-posedness and a priori estimate for the slow-fast system (\ref{1}) and averaged equation (\ref{8}).
In Sec. 4, we derive the stochastic averaging principle in sense of strong convergence for (\ref{1}) by using the Khasminskii technique.
\section{Preliminary results}

\subsection{Green¡¯s function for wave equation}
For the deterministic
wave equation
\begin{eqnarray*}
\begin{array}{l}
\begin{array}{llll}
u_{tt}-u_{xx}=0,
\end{array}
\end{array}
\end{eqnarray*}
its \textbf{Green¡¯s function} is given by
\begin{eqnarray*}
\begin{array}{l}
\begin{array}{llll}
K(t,\xi,\zeta)=\sum\limits_{k=1}^{\infty}\frac{\sin(\sqrt{\alpha_{k}t})}{\sqrt{\alpha_{k}}}e_{k}(\xi)e_{k}(\zeta).
\end{array}
\end{array}
\end{eqnarray*}
It is easy to shown that the above series converge in $L^{2}(I\times I)$ and the associated
\textbf{Green¡¯s operator} is defined by, for any $h(\xi)\in L^{2}(I),$
\begin{eqnarray*}
\begin{array}{l}
\begin{array}{llll}
G(t)h(\xi)=\int_{I}K(t,\xi,\zeta)h(\zeta)d\zeta=\sum\limits_{k=1}^{\infty}\frac{\sin(\sqrt{\alpha_{k}}t)}{\sqrt{\alpha_{k}}}e_{k}(\xi)(e_{k},h).
\end{array}
\end{array}
\end{eqnarray*}
\par
For Green operator $G(t)$, it is easy to derive the following results:
\begin{lemma}\label{L9}\cite[P133, Lemma 3.1, Lemma 3.2]{C4} Green operator $G(t)$ satisfies
\par
1) Let $k$ and $m$ be nonnegative integers. Then, for any function
$h\in H^{k+m-1}$, the following estimates hold:
\begin{eqnarray*}
\begin{array}{l}
\begin{array}{llll}
\sup_{0\leq t\leq T}\|G^{(k)}(t)h\|_{H^{m}}^{2}\leq \|h\|_{H^{k+m-1}}^{2},
~~
{\rm{for}}~~0\leq k+m\leq2.
\end{array}
\end{array}
\end{eqnarray*}
\par
2) Let $f(\cdot,t)\in L^{2}(\Omega\times(0,T);L^{2}(I))$ satisfy
\begin{eqnarray*}
\begin{array}{l}
\begin{array}{llll}
\mathbb{E}\int_{0}^{T}\|f(\cdot,t)\|^{2}dt<\infty.
\end{array}
\end{array}
\end{eqnarray*}
Then
\begin{eqnarray*}
\begin{array}{l}
\begin{array}{llll}
\int_{0}^{t}G(t-s)f(\cdot,s)ds
\end{array}
\end{array}
\end{eqnarray*} is a continuous, adapted $H^{1}$-valued process and its time
derivative is a continuous $L^{2}(I)$-valued process such that
\begin{eqnarray*}
\begin{array}{l}
\begin{array}{llll}
\mathbb{E}\sup_{0\leq t\leq T}\|\int_{0}^{t}G(t-s)f(\cdot,s)ds\|_{H^{k}}^{2}\leq C_{k}T\mathbb{E}\int_{0}^{T}\|f(\cdot,s)\|_{H^{k-1}}^{2}ds,
~~
{\rm{for}}~~k=0,1,

\end{array}
\end{array}
\end{eqnarray*}
and
\begin{eqnarray*}
\begin{array}{l}
\begin{array}{llll}
\mathbb{E}\sup_{0\leq t\leq T}\|(\int_{0}^{t}G(t-s)f(\cdot,s)ds)^{\prime}\|^{2}\leq T\mathbb{E}\int_{0}^{T}\|f(\cdot,s)\|^{2}ds,
~~
{\rm{for}}~~k=0,1,

\end{array}
\end{array}
\end{eqnarray*}
\end{lemma}
\par
According to Lemma \ref{L9}, we have
\begin{corollary}\label{C1} Green operator $G(t)$ satisfies:
for any $p>0,$
\par
1) Let $k$ and $m$ be nonnegative integers. Then, for any function
$h\in H^{k+m-1}$, the following estimates hold:
\begin{eqnarray*}
\begin{array}{l}
\begin{array}{llll}
\sup_{0\leq t\leq T}\|G^{(k)}(t)h\|_{H^{m}}^{p}\leq \|h\|_{H^{k+m-1}}^{p},
~~
{\rm{for}}~~0\leq k+m\leq2.
\end{array}
\end{array}
\end{eqnarray*}
\par
2) Let $f(\cdot,t)\in L^{p}(\Omega\times(0,T);L^{2}(I))$ satisfy
\begin{eqnarray*}
\begin{array}{l}
\begin{array}{llll}
\mathbb{E}\int_{0}^{T}\|f(\cdot,t)\|^{p}dt<\infty.
\end{array}
\end{array}
\end{eqnarray*}
Then
\begin{eqnarray*}
\begin{array}{l}
\begin{array}{llll}
\int_{0}^{t}G(t-s)f(\cdot,s)ds
\end{array}
\end{array}
\end{eqnarray*} is a continuous, adapted $H^{1}$-valued process and its time
derivative is a continuous $L^{2}(I)$-valued process such that
\begin{eqnarray*}
\begin{array}{l}
\begin{array}{llll}
\mathbb{E}\sup_{0\leq t\leq T}\|\int_{0}^{t}G(t-s)f(\cdot,s)ds\|_{H^{k}}^{p}\leq C_{k}T^{p-1}\mathbb{E}\int_{0}^{T}\|f(\cdot,s)\|_{H^{k-1}}^{p}ds,
~~
{\rm{for}}~~k=0,1,

\end{array}
\end{array}
\end{eqnarray*}
and
\begin{eqnarray*}
\begin{array}{l}
\begin{array}{llll}
\mathbb{E}\sup_{0\leq t\leq T}\|(\int_{0}^{t}G(t-s)f(\cdot,s)ds)^{\prime}\|^{p}\leq T^{p-1}\mathbb{E}\int_{0}^{T}\|f(\cdot,s)\|^{p}ds,
~~
{\rm{for}}~~k=0,1,

\end{array}
\end{array}
\end{eqnarray*}
\end{corollary}
\subsection{The heat semigroup $\{S(t)\}_{t\geq 0}$ }
\par
According to \cite[P83]{Y1}, the operator $-\mathcal{L}$ is positive, self-adjoint and sectorial on the
domain $\mathcal{D}(-\mathcal{L})=H^{2}(I)\cap H_{0}^{1}(I)$. By spectral theory, we may define
the fractional powers $(-\mathcal{L})^{\alpha}$ of $-\mathcal{L}$ with the domain $\mathcal{D}((-\mathcal{L})^{\alpha})$ for any
$\alpha\in [0,1]$. We know that the semigroup $\{S(t)\}_{t\geq 0}$ generated by the
operator $-\mathcal{L}$ is analytic on $L^{p}(I)$ for all $1\leq p\leq\infty$ and enjoys the
following properties \cite{P1}:
\begin{eqnarray}\label{13}
\begin{array}{l}
\begin{array}{llll}
S(t)(-\mathcal{L})^{\alpha}=(-\mathcal{L})^{\alpha}S(t),~~~~~~~~~~~~~~~~~~~~~~~~~~~~~~~\alpha\geq0,
\\\|(-\mathcal{L})^{\alpha}S(t)\varphi\|_{L^{p}(I)}\leq Ct^{-\alpha}\|\varphi\|_{L^{p}(I)},~~~~~~~~~~~~\alpha\geq0,t\geq0,
\\
\|D^{j}S(t)\varphi\|_{L^{q}(I)}\leq Ct^{-\frac{1}{2}(\frac{1}{p}-\frac{1}{q}+j)}\|\varphi\|_{L^{p}(I)},~~q\geq p\geq1,t\geq0,
\end{array}
\end{array}
\end{eqnarray}
where $D^{j}$ denotes the $j-$th order derivative with respect to the spatial variable.

\subsection{Some useful inequalities}

\begin{lemma}\label{L7}
Let $y(t)$ be a nonnegative function, if
\begin{eqnarray*}
\begin{array}{l}
\begin{array}{llll}
y^{\prime}\leq -ay+f,

\end{array}
\end{array}
\end{eqnarray*}
we have
\begin{eqnarray*}
\begin{array}{l}
\begin{array}{llll}
y(t)\leq y(s)e^{-a(t-s)}+\int_{s}^{t}e^{-a(t-\tau)}f(\tau)d\tau.

\end{array}
\end{array}
\end{eqnarray*}
\end{lemma}

\begin{lemma}
If ~$a,b\in \mathbb{R}$, $p>0,$ it holds that
\begin{eqnarray*}
\begin{array}{l}
(|a|+|b|)^{p}\leq\left\{
\begin{array}{llll}
|a|^{p}+|b|^{p}~~~~~~~~~~~~~~~0<p\leq 1,
\\2^{p-1}(|a|^{p}+|b|^{p})~~~~~~~~~
p>1.

\end{array}
\right.
\end{array}
\end{eqnarray*}
\end{lemma}

\subsection{Some useful estimates}
\par
The following lemmas are very useful in establishing a priori estimate for the slow-fast system.
\begin{lemma}\label{L8}
Let $A_{1}$ and $A_{2}$ be two real-valued numbers and $\sigma\geq \frac{1}{2}$. Then
the following inequality is fulfilled
\begin{eqnarray*}
\begin{array}{l}
\begin{array}{llll}
||A_{1}|^{2\sigma}A_{1}-|A_{2}|^{2\sigma}A_{2}|\leq (4\sigma-1)(|A_{1}|^{2\sigma}+|A_{2}|^{2\sigma})|A_{1}-A_{2}|.
\end{array}
\end{array}
\end{eqnarray*}
\end{lemma}
\begin{remark}
The same results can be found in \cite[Lemma 7.2]{L3}.
\end{remark}

\begin{lemma}\cite[Lemma 7.3]{L3}
Let $A_{1}$ and $A_{2}$ be two real-valued numbers and $\sigma> 0$. Then
the following inequality is fulfilled
\begin{eqnarray*}
\begin{array}{l}
\begin{array}{llll}
(A_{1}-A_{2})(|A_{1}|^{2\sigma}A_{1}-|A_{2}|^{2\sigma}A_{2})\geq0.
\end{array}
\end{array}
\end{eqnarray*}
\end{lemma}
\begin{remark}
The same results can be found in \cite[Lemma 7.3]{L3}.
\end{remark}
Thus we have

\begin{corollary}\label{L1}
For any $A_{1},A_{2}\in \mathbb{R},$ we have
\begin{eqnarray*}
\begin{array}{l}
\begin{array}{llll}
(A_{1}-A_{2},\mathcal{F}(A_{1})-\mathcal{F}(A_{2}))\leq 0,
\\
(A_{1}-A_{2},\mathcal{G}(A_{1})-\mathcal{G}(A_{2}))\leq 0.
\end{array}
\end{array}
\end{eqnarray*}
\end{corollary}

\par
The following lemma is very useful in establishing a priori estimate for the slow-fast system.
\begin{lemma}\label{L5}
If $\sigma>0,$ we have
\begin{eqnarray*}
\begin{array}{l}
\begin{array}{llll}
(-A_{xx},-|A|^{2\sigma}A)\leq 0.
\end{array}
\end{array}
\end{eqnarray*}
\end{lemma}
\begin{remark}
The same results can be found in \cite[Lemma 7.4]{L3} and \cite[Lemma 2.6]{Y1}.
\end{remark}

\subsection{Preliminary results on the fast motion equation (\ref{3})}
\par
First, we consider the stochastic heat equation, the solution of (\ref{3}) will be denoted by $B^{A,B_{0}}.$
\par
We could have the following property for the solution of (\ref{3}):
\begin{lemma}\label{L6} For $A\in L^{2}(I),$ let $B^{A,X}$ be the solution of
\begin{equation}\label{14}
\begin{array}{l}
\left\{
\begin{array}{llll}
dB=[\mathcal{L}(B)+\mathcal{G}(B)+g(A,B)]dt+\sigma_{2}dW_{2}
\\B(0,t)=0=B(1,t)
\\B(x,0)=X(x)
\end{array}
\right.
\end{array}
\begin{array}{lll}
{\rm{in}}~I\times(0,+\infty)\\
{\rm{in}}~(0,+\infty)\\
{\rm{in}}~I.
\end{array}
\end{equation}
\par
1) There exists a positive constant $C$ such that $B^{A,X}$ satifies:
\begin{eqnarray}\label{22}
\begin{array}{l}
\begin{array}{llll}
\mathbb{E}\|B^{A,X}(t)\|^{2}\leq e^{-\alpha t}\|X\|^{2}+C(\|A\|^{2}+1),
\\
\mathbb{E}\|B^{A,X}(t)-B^{A,Y}(t)\|^{2}\leq \|X-Y\|^{2}e^{-2\alpha t},

\end{array}
\end{array}
\end{eqnarray}
for $t\geq0.$
\par
2) There is unique invariant measure $\mu^{A}$ for the Markov
semigroup $P_{t}^{A}$ associated with the system (\ref{14}) in $L^{2}(I).$
Moreover, we have
\begin{eqnarray*}
\int_{L^{2}(I)}\|z\|^{2}\mu^{A}(dz)\leq C(1+\|A\|^{2}).
\end{eqnarray*}
\par
3) There exists two positive constants $C$ such that $B^{A,X}$ satifies:
\begin{eqnarray*}
\begin{array}{l}
\begin{array}{llll}
\|\mathbb{E}f(A,B^{A,X})-\bar{f}(A)\|^{2}\leq C(1+\|X\|^{2}+\|A\|^{2})e^{-2\alpha t}

\end{array}
\end{array}
\end{eqnarray*}
for $t\geq0.$

\end{lemma}
\begin{proof}
1) $\bullet$ By applying the generalized It\^{o} formula with $\frac{1}{2}\|B^{A,X}\|^{2},$ we can obtain that
\begin{eqnarray*}
\begin{array}{l}
\begin{array}{llll}
\frac{1}{2}\|B^{A,X}\|^{2}=\frac{1}{2}\|X\|^{2}+\int_{0}^{t}(B^{A,X},\mathcal{L}B^{A,X}+\mathcal{G} (B^{A,X})+g(A,B^{A,X}))ds
\\~~~~~~~~~~~~~~+\int_{0}^{t}(B^{A,X},\sigma_{2}dW_{2})+\frac{1}{2}\int_{0}^{t}\|\sigma_{2}\|_{Q_{2}}^{2}ds
\\=\frac{1}{2}\|X\|^{2}-\int_{0}^{t}\|B_{x}^{A,X}\|^{2}ds+\int_{0}^{t}(B^{A,X},g(A,B^{A,X}))ds
\\~~~~~~~~~+\int_{0}^{t}(B^{A,X},\mathcal{G} (B^{A,X}))ds+\int_{0}^{t}(B^{A,X},\sigma_{2}dW_{2})+\frac{1}{2}\int_{0}^{t}\|\sigma_{2}\|_{Q_{2}}^{2}ds
.
\end{array}
\end{array}
\end{eqnarray*}
Taking mathematical expectation from both sides of above equation, we have
\begin{eqnarray*}
\begin{array}{l}
\begin{array}{llll}
\mathbb{E}\|B^{A,X}\|^{2}=\|X\|^{2}-2\int_{0}^{t}\mathbb{E}\|B_{x}^{A,X}\|^{2}ds+2\int_{0}^{t}\mathbb{E}(B^{A,X},g(A,B^{A,X}))ds
\\~~~~~~~~~+2\int_{0}^{t}\mathbb{E}(B^{A,X},\mathcal{G} (B^{A,X}))ds+\int_{0}^{t}\|\sigma_{2}\|_{Q_{2}}^{2}ds
,
\end{array}
\end{array}
\end{eqnarray*}
namely,
\begin{eqnarray*}
\begin{array}{l}
\begin{array}{llll}
\frac{d}{dt}\mathbb{E}\|B^{A,X}\|^{2}
\\=-2\mathbb{E}\|B_{x}^{A,X}\|^{2}+2\mathbb{E}(B^{A,X},g(A,B^{A,X}))
+2\mathbb{E}(B^{A,X},\mathcal{G} (B^{A,X}))+\|\sigma_{1}\|_{Q_{1}}^{2}
.
\end{array}
\end{array}
\end{eqnarray*}
According to Corollary \ref{L1}, we have
\begin{eqnarray*}
\begin{array}{l}
\begin{array}{llll}
(B^{A,X},\mathcal{G} (B^{A,X}))\leq 0,
\end{array}
\end{array}
\end{eqnarray*}
thus,
\begin{eqnarray*}
\begin{array}{l}
\begin{array}{llll}
\frac{d}{dt}\mathbb{E}\|B^{A,X}\|^{2}
\\\leq-2\mathbb{E}\|B_{x}^{A,X}\|^{2}+2\mathbb{E}(B^{A,X},g(A,B^{A,X}))+\|\sigma_{1}\|_{Q_{1}}^{2}
\\=-2\mathbb{E}\|B_{x}^{A,X}\|^{2}+2\mathbb{E}(B^{A,X},g(A,B^{A,X})-g(A,0))+2\mathbb{E}(B^{A,X},g(A,0))+\|\sigma_{1}\|_{Q_{1}}^{2}

\\\leq-2\lambda\mathbb{E}\|B^{A,X}\|^{2}+2L_{g}\mathbb{E}\|B^{A,X}\|^{2}+2\mathbb{E}(B^{A,X},g(A,0))+\|\sigma_{1}\|_{Q_{1}}^{2}
,
\end{array}
\end{array}
\end{eqnarray*}
by using the Young inequality, we have
\begin{eqnarray*}
\begin{array}{l}
\begin{array}{llll}
\frac{d}{dt}\mathbb{E}\|B^{A,X}\|^{2}
\\\leq-2\lambda\mathbb{E}\|B^{A,X}\|^{2}+2L_{g}\mathbb{E}\|B^{A,X}\|^{2}+2\mathbb{E}\|B^{A,X}\|\|g(A,0)\|+\|\sigma_{1}\|_{Q_{1}}^{2}
\\\leq-2(\lambda-L_{g})\mathbb{E}\|B^{A,X}\|^{2}+C\mathbb{E}\|B^{A,X}\|(\|A\|+1)+\|\sigma_{1}\|_{Q_{1}}^{2}
\\\leq-2(\lambda-L_{g})\mathbb{E}\|B^{A,X}\|^{2}+(\lambda-L_{g})\mathbb{E}\|B^{A,X}\|^{2}+C\|A\|^{2}+\|\sigma_{1}\|_{Q_{1}}^{2}+C
\\=-(\lambda-L_{g})\mathbb{E}\|B^{A,X}\|^{2}+C\|A\|^{2}+\|\sigma_{1}\|_{Q_{1}}^{2}+C
\\=-\alpha\mathbb{E}\|B^{A,X}\|^{2}+C\|A\|^{2}+\|\sigma_{1}\|_{Q_{1}}^{2}+C
.
\end{array}
\end{array}
\end{eqnarray*}
Hence, by applying Lemma \ref{L7} with $\mathbb{E}\|B^{A,X}\|^{2}$, we have
\begin{eqnarray*}
\begin{array}{l}
\begin{array}{llll}
\mathbb{E}\|B^{A,X}(t)\|^{2}\leq e^{-\alpha t}\|X\|^{2}+C(\|A\|^{2}+1).
\end{array}
\end{array}
\end{eqnarray*}

\par
$\bullet$
It is easy to see
\begin{eqnarray*}
\begin{array}{l}
\left\{
\begin{array}{llll}
d(B^{A,X}-B^{A,Y})=[\mathcal{L}(B^{A,X}-B^{A,Y})+\mathcal{G}(B^{A,X})-\mathcal{G}(B^{A,Y})+g(A,B^{A,X})-g(A,B^{A,Y})]dt
\\(B^{A,X}-B^{A,Y})(0,t)=0=(B^{A,X}-B^{A,Y})(1,t)
\\(B^{A,X}-B^{A,Y})(x,0)=X-Y
\end{array}
\right.
\end{array}
\begin{array}{lll}
\\
{\rm{in}}~Q\\
{\rm{in}}~(0,T)\\
{\rm{in}}~I,
\end{array}
\end{eqnarray*}
thus, it follows from the energy method that
\begin{eqnarray*}
\begin{array}{l}
\begin{array}{llll}
\frac{1}{2}\|B^{A,X}-B^{A,Y}\|^{2}

\\=\frac{1}{2}\|X-Y\|^{2}+\int_{0}^{t}(B^{A,X}-B^{A,Y},\mathcal{L}(B^{A,X}-B^{A,Y})+\mathcal{G}(B^{A,X})-\mathcal{G}(B^{A,Y})+g(A,B^{A,X})-g(A,B^{A,Y}))ds

\\=\frac{1}{2}\|X-Y\|^{2}-\int_{0}^{t}\|(B^{A,X}-B^{A,Y})_{x}\|^{2}ds+\int_{0}^{t}(B^{A,X}-B^{A,Y},\mathcal{G}(B^{A,X})-\mathcal{G}(B^{A,Y}))ds
\\~~+\int_{0}^{t}(B^{A,X}-B^{A,Y},g(A,B^{A,X})-g(A,B^{A,Y}))ds
,
\end{array}
\end{array}
\end{eqnarray*}
namely,
\begin{eqnarray*}
\begin{array}{l}
\begin{array}{llll}
\frac{d}{dt}\|B^{A,X}-B^{A,Y}\|^{2}

\\=-2\|(B^{A,X}-B^{A,Y})_{x}\|^{2}+2(B^{A,X}-B^{A,Y},\mathcal{G}(B^{A,X})-\mathcal{G}(B^{A,Y}))
\\~~~+2(B^{A,X}-B^{A,Y},g(A,B^{A,X})-g(A,B^{A,Y}))
.
\end{array}
\end{array}
\end{eqnarray*}
It follows from Lemma \ref{L1}, we have
\begin{eqnarray*}
\begin{array}{l}
\begin{array}{llll}
(B^{A,X}-B^{A,Y},\mathcal{G}(B^{A,X})-\mathcal{G}(B^{A,Y}))\leq 0,
\end{array}
\end{array}
\end{eqnarray*}
thus, we have
\begin{eqnarray*}
\begin{array}{l}
\begin{array}{llll}
\frac{d}{dt}\|B^{A,X}-B^{A,Y}\|^{2}
\\\leq-2\|(B^{A,X}-B^{A,Y})_{x}\|^{2}+2L_{g}\|B^{A,X}-B^{A,Y}\|^{2}
\\\leq-2(\lambda-L_{g})\|B^{A,X}-B^{A,Y}\|^{2}
\\=-2\alpha\|B^{A,X}-B^{A,Y}\|^{2}
,
\end{array}
\end{array}
\end{eqnarray*}
this yields
\begin{eqnarray*}
\begin{array}{l}
\begin{array}{llll}
\|B^{A,X}-B^{A,Y}\|^{2}\leq \|X-Y\|^{2}e^{-2\alpha t}
.
\end{array}
\end{array}
\end{eqnarray*}
Thus, we have
\begin{eqnarray*}
\begin{array}{l}
\begin{array}{llll}
\mathbb{E}\|B^{A,X}-B^{A,Y}\|^{2}\leq \|X-Y\|^{2}e^{-2\alpha t}
.
\end{array}
\end{array}
\end{eqnarray*}
\par
2) (\ref{22}) imply for any $A\in L^{2}(I)$ that there is unique invariant measure $\mu^{A}$ for the Markov
semigroup $P_{t}^{A}$ associated with the system (\ref{14}) in $L^{2}(I)$ such that
\begin{eqnarray*}
\begin{array}{l}
\begin{array}{llll}
\int_{L^{2}(I)}P_{t}^{A}\varphi d\mu^{A}=\int_{L^{2}(I)}\varphi d\mu^{A},~~t\geq0
\end{array}
\end{array}
\end{eqnarray*}
for any $\varphi\in B_{b}(L^{2}(I))$ the space of bounded functions on $L^{2}(I).$
\par
Then by repeating the standard argument as in \cite[Proposition 4.2]{C3} and \cite[Lemma 3.4]{C1}, the invariant
measure satisfies
\begin{eqnarray*}
\int_{L^{2}(I)}\|z\|^{2}\mu^{A}(dz)\leq C(1+\|A\|^{2}).
\end{eqnarray*}
\par
3) According to the invariant property of $\mu^{A},$ (2) and (\ref{22}), we have
\begin{eqnarray*}
\begin{array}{l}
\begin{array}{llll}
\|\mathbb{E}f(A,B^{A,X})-\bar{f}(A)\|^{2}
\\=
\|\mathbb{E}f(A,B^{A,X})-\int_{L^{2}(I)}f(A,Y)\mu ^{A}(dY)\|^{2}

\\=
\|\mathbb{E}f(A,B^{A,X})-\mathbb{E}\int_{L^{2}(I)}f(A,B^{A,Y})\mu^{A}(dY)\|^{2}
\\=
\|\int_{L^{2}(I)}\mathbb{E}[f(A,B^{A,X})-f(A,B^{A,Y})]\mu^{A}(dY)\|^{2}
\\\leq
C\int_{L^{2}(I)}\mathbb{E}\|B^{A,X}-B^{A,Y}\|^{2}\mu^{A}(dY)
\\\leq C\int_{L^{2}(I)}\|X-Y\|^{2}e^{-2\alpha t}\mu^{A}(dY)
\\\leq C(1+\|X\|^{2}+\|A\|^{2})e^{-2\alpha t}
.
\end{array}
\end{array}
\end{eqnarray*}

\end{proof}

\section{Well-posedness and a priori estimate for the slow-fast system (\ref{1}) and averaged equation (\ref{8})}
\par
We first establish the well-posedness for the slow-fast system (\ref{1}). We consider the mild solution of (\ref{1}). The Banach
contraction principle is used as the main tool for proving the existence of mild solutions of SPDE
in most of the existing papers. We first apply the fixed point theorem to the corresponding
truncated equation and give the local existence of mild solutions to (\ref{1}). Then, the energy
estimate shows that the solution is also global in time.
\subsection{Well-posedness and a priori estimate for the slow-fast system (\ref{1})}
\begin{definition}
If $(A^{\varepsilon},B^{\varepsilon})$ is an adapted process over $(\Omega,\mathcal {F},\{\mathcal {F}_{t}\}_{t\geq0},\mathbb{P})$ such that $\mathbb{P}-$a.s. the
integral equations
\begin{eqnarray*}
\begin{array}{l}
\begin{array}{llll}

A^{\varepsilon}(t)=G^{\prime}(t)A_{0}+G(t)A_{1}+\int_{0}^{t}G(t-s)[\mathcal{F}(A^{\varepsilon})+f(A^{\varepsilon}, B^{\varepsilon})](s) ds+\int_{0}^{t}G(t-s)\sigma_{1}dB_{1}
\\B^{\varepsilon}(t)=S(\frac{t}{\varepsilon})B_{0}+\frac{1}{\varepsilon}\int_{0}^{t}S(\frac{t-s}{\varepsilon})[\mathcal{G}(B^{\varepsilon})+g(A^{\varepsilon}, B^{\varepsilon})](s)ds+\frac{1}{\sqrt{\varepsilon}}\int_{0}^{t}S(\frac{t-s}{\varepsilon})\sigma_{2}dB_{2}
\end{array}
\end{array}
\end{eqnarray*}
hold true for all $t>0,$ we say that it is a mild solution for
Eqs. (\ref{1}).
\end{definition}

\begin{proposition}\label{P6}
For any $\varepsilon\in(0,1),T>0,p\geq1,$ if $A_{0},B_{0}\in H^{1}_{0}(I),A_{1}\in L^{2}(I),$ (\ref{1}) admits a unique mild solution $(A^{\varepsilon},B^{\varepsilon})\in X_{2,T}.$
\end{proposition}
\par
The proof of well-posedness for the slow-fast system (\ref{1}) is divided into
several steps.
\subsubsection{ Local existence}
We can establish the local well-posedness for the slow-fast system (\ref{1}) in $X_{p,T}(p\geq 1).$
\begin{lemma}\label{P5}
For any $A_{0},B_{0}\in H^{1}_{0}(I),A_{1}\in L^{2}(I),$ and $p\geq1,$ $\varepsilon\in(0,1)$ (\ref{1}) admits a unique mild solution $(A^{\varepsilon},B^{\varepsilon})\in X_{p,\tau_{\infty}},$ where $\tau_{\infty}$ is
stopping time for $p.$ Moreover, if $\tau_{\infty}<+\infty,$ then $\mathbb{P}-$a.s.
\begin{eqnarray*}
\begin{array}{l}
\begin{array}{llll}
\lim\limits_{t\rightarrow \tau_{\infty}}\|(A^{\varepsilon},B^{\varepsilon})\|_{Y_{t}}=+\infty.
\end{array}
\end{array}
\end{eqnarray*}
\end{lemma}
\begin{proof}

Inspired from \cite{L1}, let $\rho\in C^{\infty}_{0}(\mathbb{R})$ be a cut-off function such that $\rho(r)=1$ for $r\in[0,1]$ and $\rho(r)=0$ for $r\geq 2.$ For any $R>0,y\in X_{p,t}$ and $t\in [0,T],$ we set
\begin{eqnarray*}
\begin{array}{l}
\begin{array}{llll}
\rho_{R}(y)(t)=\rho(\frac{\|y\|_{C([0,t];H^{1}(I))}}{R}).
\end{array}
\end{array}
\end{eqnarray*}
The truncated equation corresponding to (\ref{1}) is the
following stochastic partial differential equation:
\begin{equation*}
\begin{array}{l}
\left\{
\begin{array}{llll}
dA^{\varepsilon}_{t}=[\mathcal{L}(A^{\varepsilon})+\rho_{R}(A^{\varepsilon})\mathcal{F}(A^{\varepsilon})+f(A^{\varepsilon}, B^{\varepsilon})]dt+\sigma_{1}dW_{1}
\\dB^{\varepsilon}=\frac{1}{\varepsilon}[\mathcal{L}(B^{\varepsilon})+\rho_{R}(B^{\varepsilon})\mathcal{G}(B^{\varepsilon})+g(A^{\varepsilon}, B^{\varepsilon})]dt+\frac{1}{\sqrt{\varepsilon}}\sigma_{2}dW_{2}
\\A^{\varepsilon}(0,t)=0=A^{\varepsilon}(1,t)
\\B^{\varepsilon}(0,t)=0=B^{\varepsilon}(1,t)
\\A^{\varepsilon}(x,0)=A_{0}(x)
\\A^{\varepsilon}_{t}(x,0)=A_{1}(x)
\\B^{\varepsilon}(x,0)=B_{0}(x)
\end{array}
\right.
\end{array}
\begin{array}{lll}
{\rm{in}}~Q,\\
{\rm{in}}~Q,\\
{\rm{in}}~(0,T),\\
{\rm{in}}~(0,T),\\
{\rm{in}}~I,\\
{\rm{in}}~I,\\
{\rm{in}}~I.
\end{array}
\end{equation*}
\par
In this proof, we will take
\begin{equation*}
\begin{array}{l}
\begin{array}{llll}
\varepsilon=1
\end{array}
\end{array}
\end{equation*}
for the sake of simplicity. All the results
can be extended without difficulty to the general case.\
Thus, we consider the following system
\begin{equation*}
\begin{array}{l}
\left\{
\begin{array}{llll}
dA_{t}=[\mathcal{L}(A)+\rho_{R}(A)\mathcal{F}(A)+f(A, B)]dt+\sigma_{1}dW_{1}
\\dB=[\mathcal{L}(B)+\rho_{R}(B)\mathcal{G}(B)+g(A, B)]dt+\sigma_{2}dW_{2}
\\A(0,t)=0=A(1,t)
\\B(0,t)=0=B(1,t)
\\A(x,0)=A_{0}(x)
\\A_{t}(x,0)=A_{1}(x)
\\B(x,0)=B_{0}(x)
\end{array}
\right.
\end{array}
\begin{array}{lll}
{\rm{in}}~Q,\\
{\rm{in}}~Q,\\
{\rm{in}}~(0,T),\\
{\rm{in}}~(0,T),\\
{\rm{in}}~I,\\
{\rm{in}}~I,\\
{\rm{in}}~I.
\end{array}
\end{equation*}
\par
We define
\begin{eqnarray*}
\begin{array}{l}
\begin{array}{llll}
\Phi_{R}(A,B)
\\=
\left(
\begin{array}{c}\Phi_{R}^{1}(A,B)
\\\Phi_{R}^{2}(A,B)
\end{array}\right)

\\=\left(\begin{array}{c}G^{\prime}(t)A_{0}+G(t)A_{1}+\int_{0}^{t}G(t-s)[\rho_{R}(A)\mathcal{F}(A)+f(A, B)](s) ds+\int_{0}^{t}G(t-s)\sigma_{1}dB_{1}
\\S(t)B_{0}+\int_{0}^{t}S(t-s)[\rho_{R}(B)\mathcal{G}(B)+g(A, B)](s)ds+\int_{0}^{t}S(t-s)\sigma_{2}dB_{2}
\end{array}
\right).
\end{array}
\end{array}
\end{eqnarray*}
\par
$\bullet$
It is easy to see the operator $\Phi_{R}(A,B)$ maps $X_{p,T_{0}}$ into itself.
\par
$\bullet$ The estimates of
\begin{eqnarray*}
\begin{array}{l}
\begin{array}{llll}
\mathbb{E}\sup_{0\leq t\leq T_{0}}\|(\Phi_{R}^{1}(A_{1}, B_{1})-\Phi_{R}^{1}(A_{2}, B_{2}))(t)\|_{H^{1}(I)}^{p}
+\mathbb{E}\sup_{0\leq t\leq T_{0}}\|((\Phi_{R}^{1}(A_{1}, B_{1})-\Phi_{R}^{1}(A_{2}, B_{2}))(t))^{\prime}\|^{p},
\\
\mathbb{E}\sup_{0\leq t\leq T_{0}}\|(\Phi_{R}^{2}(A_{1}, B_{1})-\Phi_{R}^{2}(A_{2}, B_{2}))(t)\|_{H^{1}(I)}^{p}.
\end{array}
\end{array}
\end{eqnarray*}
\par
Indeed, due to \cite[P84]{Y1}, we have
\begin{eqnarray*}
\begin{array}{l}
\begin{array}{llll}
\|\rho_{R}(A_{1})|A_{1}|^{2\sigma}A_{1}-\rho_{R}(A_{2})|A_{2}|^{2\sigma}A_{2}\|\leq CR^{2\sigma}\|A_{1}-A_{2}\|_{H^{1}(I)}.
\end{array}
\end{array}
\end{eqnarray*}
It follows from Corollary \ref{C1} that
\begin{eqnarray}\label{16}
\begin{array}{l}
\begin{array}{llll}
\mathbb{E}\sup_{0\leq t\leq T_{0}}\|\int_{0}^{t}G(t-s)(\rho_{R}(A_{1})\mathcal{F}(A_{1})-\rho_{R}(A_{2})\mathcal{F}(A_{2}))(s)ds\|_{H^{1}(I)}^{p}
\\+\mathbb{E}\sup_{0\leq t\leq T}\|(\int_{0}^{t}G(t-s)(\rho_{R}(A_{1})\mathcal{F}(A_{1})-\rho_{R}(A_{2})\mathcal{F}(A_{2}))(s)ds)^{\prime}\|^{p}
\\\leq C T_{0}^{p-1}\mathbb{E}(\int_{0}^{T_{0}}\|(\rho_{R}(A_{1})\mathcal{F}(A_{1})-\rho_{R}(A_{2})\mathcal{F}(A_{2}))(s)\|^{p}ds)
\\\leq C T_{0}^{p-1}\mathbb{E}[\int_{0}^{T_{0}}(R^{2}\|(A_{1}-A_{2})(s)\|_{H^{1}(I)}+R^{4}\|(A_{1}-A_{2})(s)\|_{H^{1}(I)})^{p}ds]
\\\leq C T_{0}^{p-1}(R^{2p}+R^{4p})\mathbb{E}(\int_{0}^{T_{0}}\|(A_{1}-A_{2})(s)\|_{H^{1}(I)}^{p}ds)
\\\leq C T_{0}^{p}(R^{2p}+R^{4p})\mathbb{E}\sup_{0\leq t\leq T_{0}}\|(A_{1}-A_{2})(t)\|_{H^{1}(I)}^{p},
\end{array}
\end{array}
\end{eqnarray}

and
\begin{eqnarray}\label{18}
\begin{array}{l}
\begin{array}{llll}
\mathbb{E}\sup_{0\leq t\leq T_{0}}\|\int_{0}^{t}G(t-s)(f(A_{1}, B_{1})-f(A_{2}, B_{2}))(s)ds\|_{H^{1}(I)}^{p}
\\+\mathbb{E}\sup_{0\leq t\leq T}\|(\int_{0}^{t}G(t-s)(f(A_{1}, B_{1})-f(A_{2}, B_{2}))(s)ds)^{\prime}\|^{p}
\\\leq CT_{0}^{p-1}\mathbb{E}(\int_{0}^{T_{0}}\|(f(A_{1}, B_{1})-f(A_{2}, B_{2}))(s)\|^{p}ds)

\\\leq CT_{0}^{p-1}\mathbb{E}(\int_{0}^{T_{0}}(\|(A_{1}-A_{2})(s)\|^{p}+\|(B_{1}-B_{2})(s)\|^{p})ds)

\\\leq C T_{0}^{p}(\mathbb{E}\sup_{0\leq t\leq T_{0}}\|(A_{1}-A_{2})(t)\|^{p}+\mathbb{E}\sup_{0\leq t\leq T_{0}}\|(B_{1}-B_{2})(t)\|^{p}).
\end{array}
\end{array}
\end{eqnarray}
Finally, collecting the above estimates (\ref{16})-(\ref{18}), we get
\begin{eqnarray}\label{19}
\begin{array}{l}
\begin{array}{llll}
\mathbb{E}\sup_{0\leq t\leq T_{0}}\|(\Phi_{R}^{1}(A_{1}, B_{1})-\Phi_{R}^{1}(A_{2}, B_{2}))(t)\|_{H^{1}(I)}^{p}
\\~~~+\mathbb{E}\sup_{0\leq t\leq T_{0}}\|((\Phi_{R}^{1}(A_{1}, B_{1})-\Phi_{R}^{1}(A_{2}, B_{2}))(t))^{\prime}\|^{p}
\\\leq C[T_{0}^{p}(R^{2p}+R^{4p})+T_{0}^{p}]
(\mathbb{E}\sup_{0\leq t\leq T_{0}}\|(A_{1}-A_{2})(t)\|_{H^{1}(I)}^{p}+\mathbb{E}\sup_{0\leq t\leq T_{0}}\|(B_{1}-B_{2})(t)\|_{H^{1}(I)}^{p}).
\end{array}
\end{array}
\end{eqnarray}

\par
By taking $p=q=2,j=1$ in the third inequality of (\ref{13}), we have
\begin{eqnarray}\label{6}
\begin{array}{l}
\begin{array}{llll}
\mathbb{E}\sup_{0\leq t\leq T_{0}}\|\int_{0}^{t}S(t-s)(\rho_{R}(B_{1})\mathcal{G}(B_{1})-\rho_{R}(A_{2})\mathcal{G}(B_{2}))(s)ds\|_{H^{1}}^{p}
\\\leq C\mathbb{E}\sup_{0\leq t\leq T_{0}}(\int_{0}^{t}(t-s)^{-\frac{1}{2}}\|(\rho_{R}(B_{1})\mathcal{G}(B_{1})-\rho_{R}(B_{2})\mathcal{G}(B_{2}))(s)\|ds)^{p}
\\\leq C\mathbb{E}\sup_{0\leq t\leq T_{0}}(\int_{0}^{t}(t-s)^{-\frac{1}{2}}R^{2}\|(B_{1}-B_{2})(s)\|_{H^{1}}ds)^{p}
\\\leq C R^{2p}\sup_{0\leq t\leq T_{0}}(\int_{0}^{t}(t-s)^{-\frac{1}{2}}ds)^{p}\mathbb{E}\sup_{0\leq t\leq T_{0}}\|(B_{1}-B_{2})(t)\|_{H^{1}}^{p}
\\\leq C R^{2p}T_{0}^{\frac{p}{2}}\mathbb{E}\sup_{0\leq t\leq T_{0}}\|(B_{1}-B_{2})(t)\|_{H^{1}}^{p},
\end{array}
\end{array}
\end{eqnarray}
and
\begin{eqnarray}\label{20}
\begin{array}{l}
\begin{array}{llll}
\mathbb{E}\sup_{0\leq t\leq T_{0}}\|\int_{0}^{t}S(t-s)(g(A_{1}, B_{1})-g(A_{2}, B_{2}))(s)ds\|_{H^{1}(I)}^{p}
\\\leq \mathbb{E}\sup_{0\leq t\leq T_{0}}(\int_{0}^{t}\|S(t-s)(g(A_{1}, B_{1})-g(A_{2}, B_{2}))(s)\|_{H^{1}(I)}ds)^{p}
\\\leq C\mathbb{E}\sup_{0\leq t\leq T_{0}}(\int_{0}^{t}(t-s)^{-\frac{1}{2}}\|(g(A_{1}, B_{1})-g(A_{2}, B_{2}))(s)\|ds)^{p}
\\\leq C\mathbb{E}\sup_{0\leq t\leq T_{0}}(\int_{0}^{t}(t-s)^{-\frac{1}{2}}(\|(A_{1}-A_{2})(s)\|+\|(B_{1}-B_{2})(s)\|)ds)^{p}
\\\leq C \sup_{0\leq t\leq T_{0}}(\int_{0}^{t}(t-s)^{-\frac{1}{2}}ds)^{p}(\mathbb{E}\sup_{0\leq t\leq T_{0}}\|(A_{1}-A_{2})(t)\|^{p}+\mathbb{E}\sup_{0\leq t\leq T_{0}}\|(B_{1}-B_{2})(t)\|^{p})
\\\leq C T_{0}^{\frac{p}{2}}(\mathbb{E}\sup_{0\leq t\leq T_{0}}\|(A_{1}-A_{2})(t)\|^{p}+\mathbb{E}\sup_{0\leq t\leq T_{0}}\|(B_{1}-B_{2})(t)\|^{p}).
\end{array}
\end{array}
\end{eqnarray}
According to (\ref{6}) and (\ref{20}), we have
\begin{eqnarray*}
\begin{array}{l}
\begin{array}{llll}
\mathbb{E}\sup_{0\leq t\leq T_{0}}\|(\Phi_{R}^{2}(A_{1}, B_{1})-\Phi_{R}^{2}(A_{2}, B_{2}))(t)\|_{H^{1}(I)}^{p}
\\\leq C(R^{2p}T_{0}^{\frac{p}{2}}+T_{0}^{\frac{p}{2}})
(\mathbb{E}\sup_{0\leq t\leq T_{0}}\|(A_{1}-A_{2})(t)\|_{H^{1}(I)}^{p}+\mathbb{E}\sup_{0\leq t\leq T_{0}}\|(B_{1}-B_{2})(t)\|_{H^{1}(I)}^{p}),
\end{array}
\end{array}
\end{eqnarray*}
\par
It follows from (\ref{19}) and (\ref{20}) that
\begin{eqnarray*}
\begin{array}{l}
\begin{array}{llll}
\mathbb{E}\sup_{0\leq t\leq T_{0}}\|(\Phi_{R}^{1}(A_{1}, B_{1})-\Phi_{R}^{1}(A_{2}, B_{2}))(t)\|_{H^{1}(I)}^{p}
\\+\mathbb{E}\sup_{0\leq t\leq T_{0}}\|((\Phi_{R}^{1}(A_{1}, B_{1})-\Phi_{R}^{1}(A_{2}, B_{2}))(t))^{\prime}\|^{p}
\\+\mathbb{E}\sup_{0\leq t\leq T_{0}}\|(\Phi_{R}^{2}(A_{1}, B_{1})-\Phi_{R}^{2}(A_{2}, B_{2}))(t)\|_{H^{1}(I)}^{p}
\\\leq C(T_{0}^{p}(R^{2p}+R^{4p})+T_{0}^{p}+T_{0}^{\frac{p}{2}}R^{2p}+T_{0}^{\frac{p}{2}})
(\mathbb{E}\sup_{0\leq t\leq T_{0}}\|(A_{1}-A_{2})(t)\|_{H^{1}(I)}^{p}+\mathbb{E}\sup_{0\leq t\leq T_{0}}\|(B_{1}-B_{2})(t)\|_{H^{1}(I)}^{p}),
\end{array}
\end{array}
\end{eqnarray*}
namely, we have
\begin{eqnarray}\label{21}
\begin{array}{l}
\begin{array}{llll}
\|\Phi_{R}(A_{1}, B_{1})-\Phi_{R}(A_{2}, B_{2})\|_{X_{p,T_{0}}}
\\\leq C(T_{0}(R^{2}+R^{4})+T_{0}+T_{0}^{\frac{1}{2}}R^{2}+T_{0}^{\frac{1}{2}})\|(A_{1},B_{1})-(A_{2},B_{2})\|_{X_{p,T_{0}}}.
\end{array}
\end{array}
\end{eqnarray}
\par
$\bullet$
For a sufficiently small $T_{0},$ is $\Phi_{R}(A,B)$ a contraction mapping on $X_{p,T_{0}}.$
\par
Hence, by applying the Banach contraction principle, $\Phi_{R}(A,B)$ has a unique fixed point in $X_{p,T_{0}},$
which is the unique local solution to (\ref{1}) on the interval
$[0,T_{0}].$ Since $T_{0}$ does not depend on the initial value $(A_{0},B_{0}),$ this
solution may be extended to the whole interval $[0,T].$
\par
We denote by $(A_{R},B_{R})$ this unique mild solution and let
\begin{eqnarray*}
\begin{array}{l}
\begin{array}{llll}
\tau_{R}=\inf\{t\geq0:\|(A_{R},B_{R})\|_{X_{p,t}}\geq R\},
\end{array}
\end{array}
\end{eqnarray*}
with the usual convention that $\inf \emptyset=\infty.$
\par
Since $R_{1}\leq R_{2},$ $\tau_{{R}_{1}}\leq \tau_{{R}_{2}},$ we can put $\tau_{\infty}=\lim\limits_{R\rightarrow +\infty}\tau_{R}.$
We define a local solution to (\ref{1}) as follows
\begin{eqnarray*}
\begin{array}{l}
\begin{array}{llll}
A(t)=A_{R}(t),~\forall t\in [0,\tau_{R}],
\\B(t)=B_{R}(t),~\forall t\in [0,\tau_{R}].
\end{array}
\end{array}
\end{eqnarray*}
Indeed, for any $t\in [0,\tau_{R_{1}}\wedge\tau_{R_{2}}]$
\begin{eqnarray*}
\begin{array}{l}
\begin{array}{llll}
~~~~A_{R_{1}}(t)-A_{R_{2}}(t)
\\=\int_{0}^{t}G(t-s)(\rho_{R_{1}}(A_{R_{1}})\mathcal{F}(A_{R_{1}})-\rho_{R_{2}}(A_{R_{2}})\mathcal{F}(A_{R_{2}})
+f(A_{R_{1}}, B_{R_{1}})-f(A_{R_{2}}, B_{R_{2}})) (s) ds,
\\
~~~~B_{R_{1}}(t)-B_{R_{2}}(t)
\\=\int_{0}^{t}S(t-s)(\rho_{R_{1}}(B_{R_{1}})\mathcal{G}(B_{R_{1}})-\rho_{R_{2}}(B_{R_{2}})\mathcal{G}(B_{R_{2}})+g(A_{R_{1}}, B_{R_{1}})-g(A_{R_{2}}, B_{R_{2}})) (s) ds.
\end{array}
\end{array}
\end{eqnarray*}
Proceeding as in the proof of (\ref{21}), we can obtain
\begin{eqnarray*}
\begin{array}{l}
\begin{array}{llll}
\|(A_{R_{1}},B_{R_{1}})-(A_{R_{2}},B_{R_{2}})\|_{X_{p,t}}
\\\leq C(t)\|(A_{R_{1}},B_{R_{1}})-(A_{R_{2}},B_{R_{2}})\|_{X_{p,t}},
\end{array}
\end{array}
\end{eqnarray*}
where $C(t)$ is a monotonically increasing function and $C(0)=0.$
If we take $t$ sufficiently small, we can obtain
\begin{eqnarray*}
\begin{array}{l}
\begin{array}{llll}
A_{R_{1}}(t)=A_{R_{2}}(t),
\\B_{R_{1}}(t)=B_{R_{2}}(t).
\end{array}
\end{array}
\end{eqnarray*}
Repeating the same argument for the interval $[t,2t]$ and so on yields
\begin{eqnarray*}
\begin{array}{l}
\begin{array}{llll}
A_{R_{1}}(t)=B_{R_{2}}(t),
\\A_{R_{1}}(t)=B_{R_{2}}(t).
\end{array}
\end{array}
\end{eqnarray*}
 for the whole interval $[0,\tau].$
According to this, we can know the above definition of local solution to (\ref{1}) is well defined.
\par
If $\tau_{\infty}<+\infty,$ the definition of $(A,B)$ yields $\mathbb{P}-$a.s.
\begin{eqnarray*}
\begin{array}{l}
\begin{array}{llll}
\lim\limits_{t\rightarrow \tau_{\infty}}\|(A,B)\|_{X_{p,t}}=+\infty,
\end{array}
\end{array}
\end{eqnarray*}
which shows that $(A,B)$ is a unique local solution to (\ref{1}) on the
interval $[0,\tau_{\infty}).$
\par
This completes the proof of Lemma \ref{P5}.
\end{proof}

\subsubsection{Energy inequalities for the slow-fast system (\ref{1})}
\par
Now, we establish some energy inequalities for the slow-fast system (\ref{1}).
\begin{proposition}\label{P2}
Let $\xi=\inf\{\tau_{\infty},T\}.$ If $A_{0},B_{0}\in H^{1}_{0}(I),A_{1}\in L^{2}(I),$ for $\varepsilon\in(0,1),$ $(A^{\varepsilon},B^{\varepsilon})$ is the unique solution to (\ref{1}),  then there exists a constant $C$ such that the solutions $(A^{\varepsilon},B^{\varepsilon})$ satisfy
\begin{eqnarray*}
\begin{array}{l}
\begin{array}{llll}
\sup\limits_{\varepsilon\in (0,1)}\mathbb{E}\sup\limits_{0\leq t\leq \xi}(\|A^{\varepsilon}_{x}(t)\|^{2}+\|A^{\varepsilon}_{t}(t)\|^{2}+\|A^{\varepsilon}(t)\|_{L^{4}(I)}^{4}+\|A^{\varepsilon}(t)\|_{L^{6}(I)}^{6})\leq C,
\\
\sup\limits_{\varepsilon\in (0,1)}\sup\limits_{t\in[0,\xi]}\mathbb{E}\|A^{\varepsilon}(t)\|_{H^{1}(I)}^{2}\leq C,
\\
\mathbb{E}\sup\limits_{t\in[0,\xi]}\|B^{\varepsilon}(t)\|_{H^{1}(I)}^{2}\leq \frac{C}{\varepsilon},
\\
\sup\limits_{\varepsilon\in (0,1)}\mathbb{E}\int_{0}^{\xi}\|B^{\varepsilon}_{xx}\|^{2}dt
\leq C.
\end{array}
\end{array}
\end{eqnarray*}
where $C$ dependent of $T,A_{0},B_{0}$ but independent of $\varepsilon\in (0,1).$

\end{proposition}

\begin{proof}
The proof of Proposition \ref{P2} is divided into
several steps.
\par
$\bullet$ The estimates of $\mathbb{E}\sup\limits_{0\leq t\leq \xi}(\|A^{\varepsilon}_{x}(t)\|^{2}+\|A^{\varepsilon}_{t}(t)\|^{2}+\|A^{\varepsilon}(t)\|_{L^{4}(I)}^{4}+\|A^{\varepsilon}(t)\|_{L^{6}(I)}^{6})$ and $\sup\limits_{0\leq t\leq \xi}\mathbb{E}\|B^{\varepsilon}(t)\|^{2}.$
\par
$\star$ Indeed, it follows from \cite[P137, Theorem 3.5]{C4} that
\begin{eqnarray*}

\end{eqnarray*}
it is easy to see
\begin{eqnarray*}
%
\end{eqnarray*}
this implies that
\begin{eqnarray*}
%
\end{eqnarray*}
\par
By taking mathematical expectation from both sides of above equation, we have
\begin{eqnarray*}
%
\end{eqnarray*}

In view of the Burkholder-Davis-Gundy inequality, it holds that
\begin{eqnarray*}
%
\end{eqnarray*}
In view of the H\"{o}lder inequality, it holds that
\begin{eqnarray*}
%
\end{eqnarray*}

According to the above estimates, we have
\begin{eqnarray*}
%
\end{eqnarray*}
by taking $0<\eta<<1,$ it holds that
\begin{eqnarray*}
%
\end{eqnarray*}
thus, it follows from Gronwall inequality that
\begin{eqnarray*}
%
\end{eqnarray*}
moreover, we have
\begin{eqnarray}\label{7}
%
\end{eqnarray}

\par
$\star$ Indeed, we apply the generalized It\^{o} formula with $\|B^{\varepsilon}\|^{2}$ and obtain that
\begin{eqnarray*}
%
\end{eqnarray*}
this implies that
\begin{eqnarray*}
%
\end{eqnarray*}
by taking mathematical expectation from both sides of above equation, we have
\begin{eqnarray*}
%
\end{eqnarray*}
this implies that
\begin{eqnarray*}
%
\end{eqnarray*}
We consider this term
\begin{eqnarray*}
%
\end{eqnarray*}
by using the Young inequality, we have
\begin{eqnarray*}
%
\end{eqnarray*}
it holds that
\begin{eqnarray*}
%
\end{eqnarray*}
thus, we have
\begin{eqnarray*}
%
\end{eqnarray*}
Hence, by applying Lemma \ref{L7} with $\mathbb{E}\|B^{\varepsilon}(t)\|^{2}$, we have
\begin{eqnarray*}
%
\end{eqnarray*}

Thus, plug (\ref{7}) in the above inequality, we have
\begin{eqnarray*}
%
\end{eqnarray*}
thus, it follows from Gronwall inequality that
\begin{eqnarray}\label{12}
%
\end{eqnarray}
Moreover, due to (\ref{7}) and (\ref{12}), it holds that
\begin{eqnarray}\label{15}
%
\end{eqnarray}

\par
$\bullet$ The estimate of $\sup\limits_{0\leq t\leq \xi}\mathbb{E}\|B^{\varepsilon}_{x}(t)\|^{2}.$
\par
Indeed, we apply the generalized It\^{o} formula (see \cite{Y1,C4,D1,P2}) with $\|B^{\varepsilon}_{x}\|^{2}$ and obtain that
\begin{eqnarray*}
%
\end{eqnarray*}
namely, it holds that
\begin{eqnarray}\label{5}
%
\end{eqnarray}
by taking mathematical expectation from both sides of above equation, we have
\begin{eqnarray}\label{10}
%
\end{eqnarray}
this implies that
\begin{eqnarray*}
%
\end{eqnarray*}
according to Lemma \ref{L5}, we have
\begin{eqnarray*}
%
\end{eqnarray*}
thus, it holds that
\begin{eqnarray*}
%
\end{eqnarray*}
where $\beta=\frac{\alpha(\lambda+L_{g})}{\lambda}>0.$
\par
Hence, by applying Lemma \ref{L7} with $\mathbb{E}\|B^{\varepsilon}_{x}\|^{2}$, we have
\begin{eqnarray*}
%
\end{eqnarray*}
Combining this and (\ref{15}), we have
\begin{eqnarray*}
%
\end{eqnarray*}
\par
$\bullet$ The estimate of $\mathbb{E}\sup\limits_{0\leq t\leq \xi}\|B^{\varepsilon}_{x}(t)\|^{2}.$
\par
Indeed, it follows from (\ref{5})
that
\begin{eqnarray*}
%
\end{eqnarray*}
according to Lemma \ref{L5}, we have
\begin{eqnarray*}
%
\end{eqnarray*}
thus, we have

\begin{eqnarray*}
%
\end{eqnarray*}
In view of the Burkholder-Davis-Gundy inequality and the Young inequality, it holds that
\begin{eqnarray*}
%
\end{eqnarray*}
by the Cauchy inequality, we have
\begin{eqnarray*}
%
\end{eqnarray*}
Thus, we have
\begin{eqnarray*}
%
\end{eqnarray*}
moreover, we have
\begin{eqnarray*}
%
\end{eqnarray*}
\par
$\bullet$ The estimate of $\mathbb{E}\int_{0}^{\xi}\|B^{\varepsilon}_{xx}(t)\|^{2}dt.$
\par
Indeed, it follows from (\ref{10}) that
\begin{eqnarray*}
%
\end{eqnarray*}
According to Lemma \ref{L5}, we have
\begin{eqnarray*}
%
\end{eqnarray*}
it holds that
\begin{eqnarray*}
%
\end{eqnarray*}
it is easy to see that
\begin{eqnarray*}
%
\end{eqnarray*}
thus, we have
\begin{eqnarray}\label{23}
%
\end{eqnarray}
\end{proof}
\subsubsection{Proof of Proposition \ref{P6}}
\par
Now, we prove Proposition \ref{P6}.
\begin{proof}[Proof of Proposition \ref{P6}]
By the Chebyshev inequality, Proposition \ref{P2} and the definition of $(A^{\varepsilon},B^{\varepsilon}),$ we have
\begin{eqnarray*}
%
\end{eqnarray*}
this shows that
\begin{eqnarray*}
%
\end{eqnarray*}
namely, $\tau_{\infty}=+\infty$ P-a.s.
\end{proof}

\subsubsection{Some a priori estimates for the slow-fast system (\ref{1})}
\par
Next, we establish some a priori estimates for the slow-fast system (\ref{1}).
\begin{proposition}\label{P1}
If $A_{0},B_{0}\in H^{1}_{0}(I),A_{1}\in L^{2}(I),$ for $\varepsilon\in(0,1),$ $(A^{\varepsilon},B^{\varepsilon})$ is the unique solution to (\ref{1}), then for any $p>0,$ there exists a constant $C$ such that the solutions $(A^{\varepsilon},B^{\varepsilon})$ satisfy
\begin{eqnarray*}
%
\end{array}
\end{eqnarray*}
where $C$ dependent of $p,T,A_{0},B_{0}$ but independent of $\varepsilon\in (0,1).$
\end{proposition}

\begin{proof}
The proof of Proposition \ref{P1} is divided into
several steps.
It is also suffice to prove Proposition \ref{P1} holds when $p$ is large enough. Here, the method of the proof is inspired from \cite{D2,F1,F2,F3,F4}.
\par
$\bullet$ The estimates of $\mathbb{E}\sup\limits_{0\leq t\leq T}(\|A^{\varepsilon}_{x}(t)\|^{2p}+\|A^{\varepsilon}_{t}(t)\|^{2p}+\|A^{\varepsilon}(t)\|_{L^{4}(I)}^{4p}+\|A^{\varepsilon}(t)\|_{L^{6}(I)}^{6p})$ and $\sup\limits_{0\leq t\leq T}\mathbb{E}\|B^{\varepsilon}(t)\|^{2p}.$
\par
$\star$ Indeed, it follows from \cite[P137, Theorem 3.5]{C4} that
\begin{eqnarray*}

\end{eqnarray*}
it is easy to see
\begin{eqnarray*}
%
\end{eqnarray*}
this implies that
\begin{eqnarray*}
%
\end{eqnarray*}
\par
By taking mathematical expectation from both sides of above equation, we have
\begin{eqnarray*}
%
\end{eqnarray*}

In view of the Burkholder-Davis-Gundy inequality, it holds that
\begin{eqnarray*}
%
\end{eqnarray*}
In view of the H\"{o}lder inequality, it holds that
\begin{eqnarray*}
%
\end{eqnarray*}

According to the above estimates, we have
\begin{eqnarray*}
%
\end{eqnarray*}
by taking $0<\eta<<1,$ it holds that
\begin{eqnarray*}
%
\end{eqnarray*}
thus, it follows from Gronwall inequality that
\begin{eqnarray*}
%
\end{eqnarray*}
moreover, we have
\begin{eqnarray}\label{2}
%
\end{eqnarray}

\par
$\star$ Indeed, we apply the generalized It\^{o} formula with $\|B^{\varepsilon}\|^{2p}$ and obtain that
\begin{eqnarray*}
%
\end{eqnarray*}
this implies that
\begin{eqnarray*}
%
\end{eqnarray*}
by taking mathematical expectation from both sides of above equation, we have
\begin{eqnarray*}
%
\end{eqnarray*}
this implies that
\begin{eqnarray*}
%
\end{eqnarray*}
We consider this term
\begin{eqnarray*}
%
\end{eqnarray*}
by using the Young inequality, we have
\begin{eqnarray*}
%
\end{eqnarray*}

it holds that

\begin{eqnarray*}
%
\end{eqnarray*}
thus, we have
\begin{eqnarray*}
%
\end{eqnarray*}
Hence, by applying Lemma \ref{L7} with $\mathbb{E}\|B^{\varepsilon}(t)\|^{2p}$, we have
\begin{eqnarray*}
%
\end{eqnarray*}

Thus, plug (\ref{2}) in the above inequality, we have
\begin{eqnarray*}
%
\end{eqnarray*}
thus, it follows from Gronwall inequality that
\begin{eqnarray}\label{24}
%
\end{eqnarray}
Moreover, due to (\ref{2}) and (\ref{24}), it holds that
\begin{eqnarray}\label{25}
%
\end{eqnarray}

\par
$\bullet$ The estimate of $\mathbb{E}\int_{0}^{T}\|B^{\varepsilon}_{x}(t)\|^{2p}dt.$
\par
Indeed, it follows from (\ref{5}) that
\begin{eqnarray*}
%
\end{eqnarray*}
thus, we have

\begin{eqnarray*}
%
\end{eqnarray*}
According to Lemma \ref{L5}, we have
\begin{eqnarray*}
%
\end{eqnarray*}
thus, it holds that

\begin{eqnarray*}
%
\end{eqnarray*}
Noting the fact that
\begin{eqnarray*}
%
\end{eqnarray*}
thus, it holds that

\begin{eqnarray*}
%
\end{eqnarray*}
thus, by using the Young inequality, we have

\begin{eqnarray*}
%
\end{eqnarray*}
By using the Young inequality again, we have

\begin{eqnarray*}
%
\end{eqnarray*}
by taking mathematical expectation from both sides of above equation, we have
\begin{eqnarray*}
%
\end{eqnarray*}
It follows from (\ref{23}) that
\begin{eqnarray*}
%
\end{eqnarray*}
If we take $0<\rho<<1,$ we have
\begin{eqnarray*}
%
\end{eqnarray*}
thus, it holds that
\begin{eqnarray*}
%
\end{eqnarray*}
Hence, by applying Lemma \ref{L7} with $\mathbb{E}\int_{0}^{t}\|B^{\varepsilon}_{x}\|^{2p}ds$, we have
\begin{eqnarray*}
%
\end{eqnarray*}
thus, we have
\begin{eqnarray*}
%
\end{eqnarray*}
\end{proof}

\subsection{Well-posedness for the averaged equation (\ref{8})}

\par
By the same method in Proposition \ref{P6} and Proposition \ref{P1}, we can obtain the following proposition.
\begin{proposition}\label{P3}
If $A_{0}\in H^{1}_{0}(I),A_{1}\in L^{2}_{0}(I),$ (\ref{8}) has a unique solution $\bar{A}\in L^{2}(\Omega,C([0,T];H^{1}_{0}(I))).$ Moreover, for any $p>0,$ there exists a constant $C$ such that the solution $\bar{A}$ satisfies
\begin{eqnarray*}
%
\end{eqnarray*}
where $C_{2}$ dependent of $p,T,A_{0},B_{0}$ but independent of $p>0.$
\end{proposition}

\section{Proof of Theorem \ref{Th1}}
\subsection{H\"{o}lder continuity of time variable for $A^{\varepsilon}$}
The following proposition is a crucial step.

\begin{proposition}\label{L2}
There exists a constant $C(p,T)$ such that
\begin{eqnarray}\label{11}
%
\end{eqnarray*}
we arrive at (\ref{11}).
\end{proof}

\subsection{Auxiliary process $(\hat{A}^{\varepsilon},\hat{B}^{\varepsilon})$}
\par
Next, we introduce an auxiliary process $(\hat{A}^{\varepsilon},\hat{B}^{\varepsilon})\in L^{2}(I)\times L^{2}(I)$ by Khasminskii in \cite{K1}.
\par
Fix a positive number $\delta$ and do a partition of time interval $[0,T]$ of size $\delta$. We construct a process $\hat{B}^{\varepsilon}\in L^{2}(I)$ by means of the equations
\begin{eqnarray*}
%
\end{eqnarray*}
for $t\in [k\delta,\min\{(k+1)\delta,T\}),k\geq0.$
\par
Also define
the process $\hat{A}^{\varepsilon}\in L^{2}(I)$ by
\begin{eqnarray*}
%
\end{eqnarray*}
for $t\in [0,T]$, where $s(\delta)=[\frac{s}{\delta}]\delta$ is the nearest breakpoint preceding $s$ and $[\cdot]$ is the integer function.
\par
Thus $(\hat{A}^{\varepsilon},\hat{B}^{\varepsilon})$ satisfies
\begin{equation}\label{9}
%
\end{equation}

\subsection{Some priori estimates of $(\hat{A}^{\varepsilon},\hat{B}^{\varepsilon})$}

Because the proof almost follows the steps in Proposition \ref{P1}, we omit
the proof here.

\begin{proposition}
If $A_{0},B_{0}\in H^{1}_{0}(I),$ $A_{1}\in L^{2}(I),$ for $\varepsilon\in(0,1),$ $(\hat{A}^{\varepsilon},\hat{B}^{\varepsilon})$ is the unique solution to (\ref{9}),  then there exists a constant $C$ such that the solutions $(\hat{A}^{\varepsilon},\hat{B}^{\varepsilon})$ satisfy
\begin{eqnarray*}
%
\end{eqnarray*}
where $C_{1}$ dependent of $T,A_{0},B_{0}$ but independent of $\varepsilon\in (0,1).$
\par
Moreover, for any $p>0,$ there exists a constant $C_{2}$ such that
\begin{eqnarray*}
%
\end{eqnarray*}
where $C_{2}$ dependent of $p,T,A_{0},B_{0}$ but independent of $\varepsilon\in (0,1),$ $p>0.$
\end{proposition}

\subsection{The errors of $A^{\varepsilon}-\hat{A}^{\varepsilon}$ and $B^{\varepsilon}-\hat{B}^{\varepsilon}$}
\par
We will establish convergence of the auxiliary process $\hat{B}^{\varepsilon}$ to the fast solution process $B^{\varepsilon}$ and $\hat{A}^{\varepsilon}$ to the slow solution process $A^{\varepsilon}$, respectively.
\begin{lemma}\label{L3}
There exists a constant $C$ such that
\begin{eqnarray*}
%
\end{eqnarray*}
where $C$ is only dependent of $p,T,A_{0},B_{0}.$
\end{lemma}
\begin{proof}
$\bullet$ We prove the first inequality.
\par
Indeed,
it is easy to see that $B^{\varepsilon}(t)-\hat{B}^{\varepsilon}(t)$ satisfies the following SPDE
\begin{equation}\label{4}
%
\end{eqnarray*}
By taking mathematical expectation from both sides of above equation, we have
\begin{eqnarray*}
%
\end{eqnarray*}
It follows from Lemma \ref{L1}, we have
\begin{eqnarray*}
%
\end{eqnarray*}
it holds that
\begin{eqnarray*}
%
\end{eqnarray*}
It follows from the Young inequality that
\begin{eqnarray*}
%
\end{eqnarray*}
due to Proposition \ref{L2}, it holds that
\begin{eqnarray*}
%
\end{eqnarray*}
hence, by applying Lemma \ref{L7} with $\mathbb{E}\|(B^{\varepsilon}-\hat{B}^{\varepsilon})(t)\| ^{2p}$, we have
\begin{eqnarray*}
%
\end{eqnarray*}
\par
$\bullet$ We prove the second inequality.
\par
Indeed, noting $\hat{A}^{\varepsilon},\bar{A}$ satisfy
\begin{equation*}
%
\end{equation*}
According to Lemma \ref{L9}, we have
\begin{eqnarray*}
%
\end{eqnarray*}

\par
It follows from Lemma \ref{L8}, Proposition \ref{L2} and Proposition \ref{P1} that
\begin{eqnarray*}
%
\end{eqnarray*}
by the same method, we have
\begin{eqnarray*}
%
\end{eqnarray*}
\end{proof}
\subsection{The errors of $\hat{A}^{\varepsilon}-\bar{A}$}
\par
Next we prove strong convergence of the auxiliary process $\hat{A}^{\varepsilon}$ to the averaging solution process $\bar{A}$.
\begin{lemma}\label{L4}
There exists a constant $C(T,p)$ such that
\begin{eqnarray*}
%
\end{equation*}
\par In mild sense, we introduce the following decomposition
\begin{eqnarray*}
%
\end{eqnarray*}
\par
$\bullet$ For $J_{1},$ according to Corollary \ref{C1}, we have

\begin{eqnarray*}
%
\end{eqnarray*}
we can rewrite it as
\begin{eqnarray*}
%
\end{eqnarray*}
\par
$\star$ By using the H\"{o}lder inequality, we have
\begin{eqnarray*}
%
\end{eqnarray*}
\par
It follows from Proposition \ref{P1} and Proposition \ref{L2} that
\begin{eqnarray*}
%
\end{eqnarray*}

\par
$\star$ By using the H\"{o}lder inequality and the same method in above, we have
\begin{eqnarray*}
%
\end{eqnarray*}
\par
It follows from Proposition \ref{P1} and Proposition \ref{L2} that
\begin{eqnarray*}
%
\end{eqnarray*}
\par
$\star$ By using the H\"{o}lder inequality and the same method in above, we have
\begin{eqnarray*}
%
\end{eqnarray*}
\par
In order to deal with the above estimate, we will use the skill of stopping times, this is inspired from \cite{D2}.
\par
We define the stopping time
\begin{eqnarray*}
%
\end{eqnarray*}

\par
$\bullet$ For $J_{2},$ we can rewrite $J_{2}$ as
\begin{eqnarray*}
%
\end{eqnarray*}
where $m_{t}=[\frac{t}{\delta}]$.
\par
$\star$ For $J_{21},$
it follows from \cite[P3270,Lemma 6.2]{F1} that
\begin{eqnarray*}
%
\end{eqnarray*}
\par
On the other hand, it follows from
\begin{eqnarray*}
%
\end{eqnarray*}
\par
$\star$ For $J_{22},$ due to Lemma \ref{L2}, it concludes that
\begin{eqnarray*}
%
\end{eqnarray*}
\par
$\star$ For $J_{23},$ it concludes that
\begin{eqnarray*}
%
\end{eqnarray*}
\par
With the help of the above estimates, we have
\begin{eqnarray*}
%
\end{eqnarray*}
By using the Gronwall inequality, we have
\begin{eqnarray*}
%
\end{eqnarray*}
\par
On the other hand, due to Proposition \ref{P1}, we have
\begin{eqnarray*}
%
\end{eqnarray*}
\par
Hence, we have
\begin{eqnarray*}
%
\end{eqnarray*}
if we take $n=\sqrt[4p]{-\frac{1}{8C}\ln \varepsilon},\delta=\varepsilon^{\frac{1}{2}},$ we obtain
\begin{eqnarray*}
%
\end{eqnarray*}

\end{proof}

\subsection{Proof of Theorem \ref{Th1}}
By taking $\delta=\varepsilon^{\frac{1}{2}}$ in Lemma \ref{L3}, we have
\begin{eqnarray*}
%
\end{eqnarray*}
\par
If $0<p\leq \frac{5}{8},$ for any $\kappa>0,$ it holds that
\begin{eqnarray*}
%
\end{array}
\end{eqnarray*}

\par
This completes the proof of Theorem \ref{Th1}.
\par
~~
\par
~~
\par
\noindent \footnotesize {\bf Acknowledgements.} \par I sincerely
 thank Professor Yong Li for many useful suggestions and help.\par
 \par

\baselineskip 9pt \renewcommand{\baselinestretch}{1.08}
\par
\par

\end{document}